\documentclass[a4paper]{amsart}
\usepackage{psfrag,graphicx,
}
\usepackage{amsfonts}
\usepackage{
amsmath,amscd,amsthm,amssymb,amsxtra,latexsym,epsf, 
epic,eepic,mathrsfs}
\usepackage{epsf,latexsym}
\usepackage[all]{xy}
\newfont{\Bb}{msbm10 scaled1200}
\newcommand{\basering}[1]{\ensuremath{\mathbb{#1}}}

\newcommand{\CC}{\basering{C}}
\newcommand{\HH}{\basering{H}}

\DeclareMathOperator{\OO}{\mathcal O}
\newcommand\s{\mathscr}

\begin{document}
\def\diag{\text{diag}}
\def\adj{\text{adj}}
\def\alt{\text{Alt}}
\def\vs{\vskip 10pt}
\def\sign{\text{sign }}
\def\vsn{\vskip 10pt\ni}
\def\cc{\circ}
\def\y{\item}
\def\ni{\noindent}
\def\ben{\begin{enumerate}}
\def\een{\end{enumerate}}
\def\beq{\begin{equation}}
\def\ssm{\smallsetminus}
\def\eeq{\end{equation}}
\def\bit{\begin{itemize}}
\def\eit{\end{itemize}}
\def\ec{\end{center}}
\def\bc{\begin{center}}
\def\ld{{\ldots}}
\def\cd{{\cdots}}
\def\vd{\vdots}
\def\D{{\widetilde D}}
\def\DDD{\Delta}

\def\e{\mathbf{e}}
\def\Gl{\mbox{Gl}}
\def\pgl{\mathfrak{pgl}}
\def\gl{\mathfrak{gl}}
\def\gg{\mathfrak{g}}
\def\la{\langle}
\def\ra{\rangle}


\def\th{{\theta}}
\def\ve{{\varepsilon}}
\def\pp{{\varphi}}
\def\a{{\alpha}}
\def\l{{\lambda}}
\def\rr{{\rho}}
\def\g{{\gamma}}
\def\G{{\Gamma}}
\def\o{{\omega}}
\def\O{{\Omega}}
\def\k{{\kappa}}


\def\Proof{{\bf Proof}\hspace{.2in}}
\def\eop{\hspace*{\fill}$\Box$ \vskip \baselineskip}
\def\eoq{{\hfill{$\Box$}}}

\def\bu{\bullet}
\def\tr{{\pitchfork}}
\def\b{{\bullet}}
\def\w{{\wedge}}
\def\p{{\partial}}
\def\rel{{\text{rel}}}

\def\supp{\mbox{supp}}
\def\im{\mbox{im}}
\def\Der{\mbox{Der}}
\def\wh{\widehat}
\def\Mat{\mbox{Mat}}
\def\Rep{\mbox{Rep}}
\def\sgn{\mbox{sign}}
\def\cod{\mbox{codim }\ }
\def\hh{\mbox{height }\ }
\def\opp{\mbox{\scriptsize opp}}
\def\vol{\mbox{vol}}
\def\deg{\mbox{deg}}
\def\Sym{\mbox{Sym}}
\def\Ext{\mbox{Ext}}
\def\Hom{\mbox{Hom}}
\def\End{\mbox{End}}
\def\Aut{\mbox{Aut}}
\def\DD{{\Der(\log D)}}
\def\Dh{{\Der(\log h)}}
\def\pd{\mbox{pd}}
\def\dim{{\mbox{dim}\;}}
\def\depth{{\mbox{depth}}}
\def\coker{\text{Coker}}
\def\sa{{\text{skAdj}}}
\def\adj{{\text{Adj}}}
\def\Pf{\text{Pf}}
\def\im{\text{Im}}

\def\EE{{\cal E}}
\def\R{{\cal R}}
\def\L{{\cal L}}
\def\C{{\cal C}}
\def\A{{\cal A}}
\def\TT{{\cal T}}
\newcommand{\bbbz}{{\mathbb Z}}
\def\L{\Lambda}
\def\mm{\mathfrak{m}}

\newcommand{\n}[1]{\| #1 \|}
\newcommand{\um}[1]{{\underline{#1}}}
\newcommand{\om}[1]{{\overline{#1}}}
\newcommand{\fl}[1]{\lfloor #1 \rfloor}
\newcommand{\ce}[1]{\lceil #1 \rceil}
\newcommand{\ncm}[2]
{{\left(\!\!\!\begin{array}{c}#1\\#2\end{array}\!\!\!\right)}}
\newcommand{\ncmf}[2]
{{\left[\!\!\!\begin{array}{c}#1\\#2\end{array}\!\!\!\right]}}
\newcommand{\ncms}[2]
{{\left\{\!\!\!\begin{array}{c}#1\\#2\end{array}\!\!\!\right\}}}


\def\iff{{\ \Longleftrightarrow\ }}
\def\imp{{\Rightarrow}}
\def\to{{\ \rightarrow\ }}
\def\too{{\ \longrightarrow\ }}
\def\into{{\hookrightarrow}}
\def\st{\stackrel}
\def\vs{\vskip 10pt}
\def\TTT{\Theta_S(-\log D)}
\def\TT{\Theta_S}
\def\vt{\vartheta}
\def\SS{\Sigma}
\def\ppp{\\langle\,\_\,,\_\,\rangle}
\def\wt{\widetilde}
\def\bb{(\_,\_)}

\newtheorem{theorem}{Theorem}[section]
\newtheorem{lemma}[theorem]{Lemma}
\newtheorem{sit}[theorem]{}
\newtheorem{lemmadefinition}[theorem]{Lemma and Definition}
\newtheorem{proposition}[theorem]{Proposition}
\newtheorem{example}[theorem]{Example}
\newtheorem{question}[theorem]{Question}
\newtheorem{remark}[theorem]{Remark}
\numberwithin{equation}{section}
\newtheorem{corollary}[theorem]{Corollary}
\newtheorem{definition}[theorem]{Definition}
\newtheorem{conjecture}[theorem]{Conjecture}
\def\L{\Lambda}
\def\ad{\text{adj}}
\title{
Logarithmic vector fields and the Severi strata in the discriminant} 

\author{Paul Cadman}
\address{Mathematics Institute, University of Warwick, Coventry CV4 7AL, United Kingdom}
\email{pcadman@gmail.com}
\author{David Mond}
\address{Mathematics Institute, University of Warwick, Coventry CV4 7AL, United Kingdom}
\email{d.m.q.mond@warwick.ac.uk}
\author{Duco van Straten}
\address{Institut f\"ur Mathematik,
FB 08 - Physik, Mathematik und Informatik,
Johannes Gutenberg-Universit\"at,
Staudingerweg 9, 4. OG,
55128 Mainz}
\email{straten@mathematik.uni-mainz.de}
\date{\today}

\begin{abstract}The discriminant, $D$,  in the base of a miniversal deformation of an irreducible plane curve singularity, is partitioned according to the genus of the (singular) fibre, or, equivalently, by the sum of the delta invariants of the singular points of the fibre.
The members of the partition are known as the {\it Severi strata}. The smallest is the $\delta$-constant stratum,
$D(\delta)$, where the genus of the fibre is $0$. It is well known, by work of Givental' and Varchenko, to be Lagrangian with respect to the symplectic form $\Omega$ obtained by pulling back the intersection form on the cohomology of the fibre via the period mapping. We show that the remaining Severi strata are also co-isotropic with respect to $\Omega$, and moreover that the coefficients of the expression of $\Omega^{\wedge k}$ with respect to a basis of $\Omega^{2k}(\log D)$ are equations for $D(\delta-k+1)$, for $k=1,\ld,\delta$. 
These equations
allow us to show that for $E_6$ and $E_8$, $D(\delta)$ is Cohen-Macaulay (this was already shown by Givental' for $A_{2k}$), and that, as far as we can calculate,
for $A_{2k}$ all of the Severi strata are Cohen-Macaulay. Our construction also produces a canonical rank 2 maximal Cohen Macaulay module on the discriminant. \end{abstract}

\maketitle
\section{Introduction: the discriminant and its Severi strata}\label{severi}
Two of the most basic invariants of a plane curve singularity $(C,0)$ are its 
{\em Milnor number} $\mu$ and its {\em delta invariant} $\delta$. 
If $f:(\CC^2,0)\to (\CC,0)$ is a holomorphic map defining $(C,0)=f^{-1}(0)$, 
then $\mu(C)$ is the dimension of the jacobian algebra $\OO_{\CC^2,0}/J_f$ and equals the dimension of the vanishing cohomology. 
If $n:\widetilde{C} \too C$ denotes the normalisation of $(C,0)$, 
then $\delta(C)$ is the dimension $n_*\OO_{\widetilde{C}}/\OO_{C}$ and equals the
number of double points appearing in a generic perturbation of the map $n$. 
These invariants are related by the relation
\[ \mu=2\delta+r-1\]
where $r$ denotes the number of branches of $(C,0)$ (See \cite[pp 206-211]{gls} for this relation, and for other background on curve singularities).
The number $\mu$ also appears as the number of parameters of an ${\mathscr R}_e$ miniversal deformation $F:(\CC^2\times\CC^\mu,0)\to (\CC,0)$ of the function  $f:(\CC^2,0)\to (\CC,0)$ defining $(C,0)$.
The restriction $\pi: X:=F^{-1}(0)\to S=(\CC^\mu,0)$ is a versal deformation of the plane curve singularity $(C,0)$. The fibre $C_u$ over $u \in S$ is the curve defined by zero level of the deformed function $f_u:=F(\ ,u)$ and {\em discriminant} $D \subset S$ is the set of parameter values $u$ for which the fibre $C_u$ is singular. This set is stratified by the types of singularities that appear in the fibres. In this paper we will 
focus on the so-called  {\it Severi strata}, where the sum of the delta-invariants add up to a value $ \ge k$: 
$$D(k)=\{u\in S: \delta(C_u)\geq k\}$$
where $\delta(C_u)=\sum_{x\in C_u}\delta(C_u,x)$.
Clearly $D(0)=S$ and $D(1)=D$, and as $D(i)$ is contained in $D(i-1)$
we obtain a chain 
\[D(\delta) \subset D(\delta-1) \subset \ldots \subset D(1) \subset D(0)=S\]
The smallest non-empty Severi stratum, $D(\delta)$, is the classical 
{\it $\delta$-constant stratum}. The term ``stratum'' here is a bit of a 
misnomer, since the Severi strata are not in general smooth. 

It is a classical fact, going back at least to Cayley \cite{cayley}, that any 
curve singularity with $\delta=k$ can be deformed into a curve with precisely $k$ $A_1$ points, a fact which explains the name {\em virtual number of double points} for the $\delta$-invariant. For a very nice proof see the
paper by C. Scott \cite{scott}. Thus the set $D(kA_1)$ of parameter values $u$ for which $C_u$ has precisely 
$k$ $A_1$ singularities is {\em dense} in $D(k)$. Moreover, $D(k)$ is smooth at such points, for there, by openness of versality,  $D(k)$ is a normal crossing of $k$ local smooth components of the discriminant $D$. A curve singularity with $\delta$-invariant $k>1$ is also adjacent to a curve with one $A_2$ singularity and $k-1$ $A_1$ singularities.  Hence $D(k)_{\text{\scriptsize reg}}
=D(kA_1)$. We refer to \cite{Tei80} for more background on this.

In the famous Anhang F to his {\em Vorlesungen \"uber Algebraische Geometrie} \cite{Se21}, Severi considered the variety of plane curves of degree $d$ with a given number of nodes which he used to argue for the irreducibility of the space of all curves of a given genus. A complete argument along these lines with given much later by J. Harris, \cite{harris}, and by Harris and Diaz in \cite{DH88}, which started the interest in the local case. This seems to justify the name {\em Severi-strata} for the $D(k)$'s, which was introduced in \cite{shende} . Recently, these strata have been the subject of several papers and their
geometry appears to hide some great mysteries. In \cite{FGvS} the multiplicity of 
$D(\delta)$ was shown to be equal to the Euler number of the compactified Jacobian of 
$(C,0)$. This was further explored in \cite{shende}, where multiplicities of the 
other $D(k)$ were related to the puntual Hilbert-schemes $Hilb^n(C,0)$. Most surprisingly, these invariants turn out to be related to the HOMFLY-polynomial of the knot in the $3$-sphere 
defined by $(C,0)$, \cite{shendeoblomkov}. 

If the curve $(C,0)$ is irreducible, its Milnor fibre $C_u$ has just one boundary component, and it follows that the intersection form $I_u$ on $H^1(C_u; \CC)$ is non-degenerate. In \cite{gv82}, Givental' and Varchenko 
used the period map associated to a non-degenerate $1$-form $\eta$ on the total space of the Milnor fibration of 
$F$, together with the Gauss-Manin connection,  to pull back the intersection form from the cohomology bundle 
$\s H^*$ over $S$ to get a symplectic form $\Omega$ on $S\ssm D$, and proved
\begin{theorem} \ben
\y
$\O$ extends to a symplectic form on $S$, and 
\y
the $\delta$-constant stratum $D(\delta)$ in the discriminant is Lagrangian with respect to $\O$.
\een 
\end{theorem} 

Below we complement their results and show the following theorems.
\begin{theorem}\label{th1} All of the Severi strata are coisotropic with respect to $\O$. \end{theorem}
The form $\O$ can also be used to obtain equations defining the Severi-strata. 
Let $\wedge^k \O$ be the $k$-fold wedge product of $\O$. Although it is a regular form, it can be considered as an element of $\Omega^{2k}_S(\log D)$. Let
$I_k$ be the ideal generated by its coefficients with respect to a basis of $\O^{2k}_S(\log D)$. 
\begin{theorem}\label{th2}  For $k=1,\ld, \delta$, the Severi stratum $D(k)$ is defined by the ideal $I_{\delta-k+1}$: $$D(k)=V(I_{\delta-k+1}).$$\end{theorem}

Equivalently, if $\chi_1,\ld,\chi_\mu$ form a basis for the free module of logarithmic vector fields $\Theta_S(-\log D)$, then
$D(k)$ is defined by the ideal generated by the Pfaffians of size $2\delta-2k+2$ of the skew matrix $\left(\O(\chi_i,\chi_j)\right)_{1\leq i,j,\leq \mu}$. 

We do not know whether in general the ideals $I_k$ are radical. Our calculations suggest that they are, but we have not been able to show this. 

We note that in \cite[II. Proposition 2.57]{gls} it was shown that the strata are analytic, but no equations 
were given.

Givental' proved in \cite{Giv88} that for curve singularities of type $A_{2k+1}$, $D(\delta)$ is Cohen-Macaulay and it can be conjectured that this is always the case,\cite{vS}. In the first author's PhD thesis, \cite{Cad2011},  Theorem \ref{th2} was used to show that $D(\delta)$ is Cohen Macaulay also for $E_6$ and $E_8$. Calculations using Theorem \ref{th2} suggest that the remaining Severi strata are Cohen-Macaulay in the case of $A_{2k}$, but show that for $E_6$ the stratum $D(2)$ is not Cohen-Macaulay. 

In the process of proving these theorems we noticed that $\O$ determines a natural rank $2$ maximal Cohen-Macaulay module over the discriminant $D$, which seems to be of independent interest.

\section{Preliminaries}
Let $f:(\CC^{n+1},0)\to (\CC,0)$ define an isolated singularity $(C,0)$ and let 
\[g_1,g_2,\ldots,g_{\mu}=1 \in \OO_{\CC^{n+1},0}\]
be functions giving a basis for the jacobian algebra $\OO/J_f$. We 
consider a good representative of a miniversal deformation of $f$ of the
form
\[F:B \times S\to\CC,\;\;\; F(x,u)=f(x)+\sum_{i=1}^\mu u_ig_i(x)\;,\]
where $B$ is a Milnor ball for $C$ and $S$ is a sufficiently small ball in 
$\CC^\mu$,\cite{looijenga}.
The set $X:=F^{-1}(0)$ comes with a map $\pi:X \too S$, with $C_u$ as fibre 
over $u \in S$. 

\subsection{The critical space $\Sigma$}
The relative critical set $\Sigma$ of $F$ is defined to be 
$$\Sigma=\left\{(x,u)\in B \times S:\frac{\p F}{\p x_i}(x,u)=0, i=0,\ld,n\right\}.$$
It is smooth and the projection $\pi:\Sigma\to S$ is a $\mu$-fold branched cover: its fibre over $u \in S$ consists of the critical points of $F(-,u)$. As the partial derivatives form a regular sequence,
\[\OO_\Sigma=\OO_{B\times S}/(\p F/\p x_0,\ld,\p F/\p x_n)\] is a free $\OO_S$-module of rank $\mu$. Miniversality of $F$ is equivalent to the statement that the Kodaira-Spencer map 
$$dF:\Theta_S\to \OO_{\Sigma},\quad \vartheta\mapsto \vartheta(F)=dF(\vartheta)$$ is an isomorphism.
The set
$X\cap\Sigma$ is the union over $u\in S$ of the set of singular points of $C_u$, and its image under $\pi$ is the discriminant, $D$.  For brevity we denote $X \cap \Sigma$ by $\D$. It is indeed the normalisation of $D$.
\subsection{$D$ as a free divisor}
Let $\bar F:(B\times S, (0,0))\to(\CC\times S,(0,0))$ be the unfolding of $f$ corresponding to the deformation $F$. Then $\Sigma\subset B\times \CC^\mu$ is the (absolute) critical locus of $\bar F$. We write $\Delta=\bar F(\Sigma)\subset\CC\times S$ for the set of critical values of $\bar F$. It is well known that $\Sigma$ is the normalisation of $\Delta$: it is smooth, and the map $\bar F|:\Sigma\to\Delta$ is generically one-to-one. Then $D=\Delta\cap\{0\}\times S$. As usual, $\Theta_{\CC\times S}(-\log \Delta)$ denotes the $\OO_{\CC\times S}$-module of vector fields on $\CC\times S$ which are tangent to $\Delta$, and $\Theta_S(-\log D)$ denotes the $\OO_S$-module of vector fields on $S$ which are tangent 
to $D$. 
\begin{proposition} 
(i) $\Theta_{\CC\times S}(-\log \Delta)$ is the $\OO_{\CC\times S}$-module of vector fields on $\CC\times S$ which are $\bar F$-liftable to $B\times S$. 
\newline
(ii) $\Theta_S(-\log D)$ is the $\OO_S$-module of vector fields on $S$ which are $\pi$-liftable to $V(F)$.
\end{proposition}
\begin{proof}{(\cite{looijenga})} (i) Let $\vt\in\Theta_{\CC\times S}(-\log \DDD)$. Since $F|:\Sigma\to \DDD$ is the normalisation of $\DDD$, there is a vector field $\tilde\vt_0$ on $\Sigma$ lifting $\vt$. For any extension
$\tilde\vt_1$ of $\tilde\vt_0$ to $B\times S$, $\o F(\vt)-tF(\tilde\vt_1)$ vanishes on $\Sigma$, and since the jacobian ideal $(\p F/\p x_0,\ld, \p F/\p x_{n})$ is radical, there exists a second vector field  $\xi$ on $B\times S$ such that $\o F(\tilde\vt_1)-tF(\tilde\vt_1)=tF(\xi).$ Then $\tilde\vt_1+\xi$ is an $\bar F$-lift of $\vt$. 

Conversely, suppose $\tilde\vt$ is a $\bar F$-lift of $\vt$. Then $\tilde\vt$ must be tangent to $\Sigma$, for the integral flows $\tilde\Phi_t$ and $\Phi_t$ of $\tilde\vt$ and $\vt$ satisfy $\Phi_1\circ \bar F=\bar F\circ\Phi_t$, showing that $\tilde\Phi_t$ defines an isomorphism $\bar F^{-1}(u)\to\bar F^{-1}(\Phi_t(u))$, and must therefore map singular points of $\bar F^{-1}(u)$ to singular points of $\bar F^{-1}(\Phi_t(u))$. 
It follows that $\vt$ is tangent to $\DDD$.

(ii) Let $i_0:S\to\CC\times S$ be the inclusion $u\mapsto (0,u)$. Then $D=i_0^{-1}(\DDD)$. Now $i_0$ is logarithmically transverse to $\DDD$, i.e. transverse to the distribution $\Theta_{\CC\times S}(-\log \DDD)$. If $F$ is the standard deformation $f(x)+\sum_iu_ig_i$,
with $g_\mu=1$, then this transversality is obvious: the vector field $\p/\p t+\p/\p u_\mu$ on $\CC\times S$ has $\bar F$-lift $\p/\p u_\mu$, and therefore lies in $\Theta_{\CC\times S}(-\log \DDD)$. Any other miniversal deformation is parametrised $\s R$-equivalent to the standard deformation, so the transversality holds there too. 

Identifying $\CC^\mu$ with $\{0\}\times\CC^\mu$, from the logarithmic transversality of $i_0$ to $\DDD$ it follows that $\Theta_S(-\log D)$ is equal to 
$\theta_{\CC\times S}(-\log \DDD)\bigcap\theta_{\CC\times S}(-\log (\{0\}\times S))$ restricted  to $\CC^\mu$, and that every vector field in $\Theta_S(-\log D)$ extends to a vector field in $\Theta_{\CC\times S}(-\log \DDD)$. Any lift to $\CC^{n+1}\times S$ of a vector field in $\theta_{\CC\times S}(-\log \DDD)\bigcap\theta_{\CC\times S}(-\log (\{0\}\times S))$ must be tangent to $V(F)$, and any vector field whose $\bar F$-lift is tangent to $V(F)$ must lie in $\theta_{\CC\times S}(-\log \DDD)\bigcap\theta_{\CC\times S}(-\log (\{0\}\times S))$.
 
\end{proof}
Therefore we
have a diagram 
\beq\label{logdi}\xymatrix{0&\OO_\D\ar[d]^=\ar[l]&\Theta_S\ar[d]^{dF}\ar[l]&\;\Theta_S(-\log D)\ar[d]^{\frac{dF}{F}}\ar@{_{(}->}[l]\\
0&\OO_\D\ar[l]&\OO_\Sigma\ar[l]&\OO_\Sigma\ar[l]_F&0\ar[l]}
\eeq
where the vertical maps are isomorphisms. This diagram can be used to find a basis for $\Theta_S(-\log D)$. The germs $FdF(\p/\p u_i)$ generate $(F)\OO_\Sigma$,  therefore 
if 
\beq\label{tba}
dF(\chi_i)=FdF\left(\frac{\p}{\p u_i}\right)
\eeq
then the $\chi_i$ generate $\TTT$. This shows that $\TTT$ is a locally free module, so that $D$ is a
free divisor. 

\subsection{Stratification of $D$}
The discriminant $D$ is stratified in various ways. Of special relevance to us 
are the local $\s R$ and $\s K$ strata.

Suppose as before that $F:B\times S\to\CC$ is a good representative of a versal deformation of $f$, where $B$ is open in $\CC^{n+1}$ and $S$ is open in $\CC^\mu$. Write $f_u=F(\,\_\,,u)$, and suppose that $p_1,\ld,p_k$ are the critical points of $f_u$ lying on $f_u^{-1}(0)$. For each point $p_i$,  the germ 
\[F:(B\times S,(p_i,u))\to(\CC,0)\] 
is an $\s R_e$-versal deformation of the germ
of $f_u$ at $p_i$, by openness of versality. Hence there is a germ of submersion $h_i$ from $(S,u)$ to the base of an $\s R_e$-miniversal deformation \[G_i:(B\times \CC^{\mu_i},(p_i,0))\to(\CC,0)\] 
of this germ, such that the germ of deformation $F:(B\times S, (p_i,u))\to(\CC,0)$ is isomorphic to $h_i^*( G_i)$. We set \[\s R_i(u)=h_i^{-1}(0).\] 
This is independent of the choice of miniversal deformation $G_i$ and submersion $h_i$, since any two choices can be related by a commutative diagram of spaces and maps. Again by openness of versality, the $\s R_i(u)$, $i=1,\ld, k$ are in general position with respect to one another, and we set 
\[\s R(u)=\bigcap_{i=1}^k\s R_i(u).\]  This is the {\it $\s R$ stratum through $u$}.  It is smooth of dimension $\mu-\sum_i\mu(f_u,p_i)$.    

If in the above definition we replace $F: B\times S\to \CC$ by the projection $V(F)\to S$, and
replace each $G_i$ by a $\s K_e$-miniversal deformation $H_i$ of the hypersurface singularity $(C_u,p_i)$, then we obtain the $\s K$-strata $\s K_i(u)$ and their intersection $\s K(u)$, the {\it $\s K$-stratum through $u$}, which is once again smooth, by 
openness of versality, and has dimension $\mu-\sum_i\tau(C_u,p_i)$. Since $\s R\subset\s K$, 
$\s R(u)\subset\s K(u)$.

If, for example, the fibre $C_u$ has $k$ $A_1$ singularities and no other singular points, then $\s R(u) =\s K(u)$ and its germ at $u$ coincides with the germ at $u$ of the set of points $u'$ such that $C_{u'}$ has $k$ $A_1$ points and no other singularities. 
\begin{definition}
The logarithmic tangent space $T^{\log D}_u S$ is the vector subspace of $T_uS$ spanned at $u$ 
by the values at $u$ of the germs of vector fields in $\Der(-\log D)_u$.
\end{definition}

\begin{proposition} One has the equality of vector spaces
\[T^{\log D}_uS=T_u\s K(u).\]
\end{proposition}
\begin{proof}
We have the exact sequence
$$0\to \Theta_S(-\log D) \to \Theta_S\to\pi_*(\OO_{\D})\to 0 $$
which gives
$$\frac{\Theta_S}{\Theta_S(-\log D)}\simeq \pi_*(\OO_{\D})$$
and so 
$$\frac{T_u\CC^\mu}{T^{\log D}_uS}\simeq\frac{\Theta_S}{\Theta_S(-\log D)+\mm_{S,u}\Theta_{S,u}}\simeq \bigoplus _iT^1_{\s K_e}(f_u,x_i)$$
This means that to show
$$T^{\log D}_uS=T_u\s K(u)$$ we need show only one inclusion. If $\vt\in\Theta_S(-\log D)_u$ then it has a  lift
$\wt\vt$ tangent to $V(F)$. The integral flows  of $\vt$ and $\tilde \vt$, $\pp_t$ on $(S,u)$ and $\wt\pp_t$ on $V(F)$,  satisfy $\pi\circ\tilde\pp_t=\pp_t\circ \pi.$ It follows that $\tilde\pp_t$ maps $C_u$ to $C_{\pp_t(u)}$, and therefore for each singular point $p_i$ in $C_u$, the curve germ $\{\pp_t(u):t\in(\CC,0)\}$ lies in the set
$\s K_i(u)$ defined above. Hence $\{\pp_t(u):t\in(\CC,0)\}\subset \bigcap_i\s K_i(u)=\s K(u)$, and 
$\vt(0)\in T_u\s K(u).$
\end{proof}
\def\R{\text{Res}}
\subsection{Isomorphism $\OO_{\Sigma} \to \Omega_F$}\label{symiso}
A choice of a nowhere-vanishing relative $(n+1)$-form $\o \in \Omega_{B \times S/S}^{n+1}$ 
determines an isomorphism 
\[\OO_\Sigma \simeq \O^{n+1}_F,\;\;g \mapsto g \o \] 
where 
$$\O^{n+1}_F:=\O^{n+1}_{B\times S/S}/dF\wedge\O^n_{B\times S/S}.$$

Such an isomorphism leads to many additional structures.
First of all, there is a canonical perfect pairing, the {\it residue pairing}, 
$$\R:\O^{n+1}_F\times\O^{n+1}_F\to \OO_S,$$
from which one obtains a perfect pairing on $\OO_{\Sigma}$.
$$\langle\,\_\,,\_\,\rangle:\OO_\Sigma\times\OO_\SS\to\OO_S.$$ 
Furthermore, because $\O_S^1$ and $\O^1_S(\log D)$ are $\OO_S$-dual to 
$\TT$ and $\TTT$, such a choice of $\o$ also determines isomorphisms 
$$\a:\O^1_S\to\OO_\Sigma\quad\text{ and }\quad
\beta:\Omega^1_S(\log D)\to\OO_\Sigma$$
via the formulas
$$\langle dF(\vt), \a(\xi)\rangle=\xi(\vt),\quad \text{and}\quad \langle \frac{dF}{F}(\vt),\beta(\xi)\rangle
=\xi(\vt).$$
As a result 
$\Theta_S, \Theta_S(-\log D), \O^1_S$ and $\O^1_S(\log D)$ are all identified with 
$\OO_\Sigma$ and hence with one another. For example we have the isomorphism 
$$k^{-1}\circ\beta:\O^1_S(\log D)\to \TT,$$ where $k:\TT\to\OO_\SS$ is the Kodaira-Spencer map $dF$. 

Note that for any $a,b,c\in \OO_\Sigma$, the pairing satisfies 
$$\langle a,bc\rangle =\langle ab,c\rangle,$$
and so multiplication by $F$ on $\OO_\SS$ is self-adjoint:
$$\langle g, Fh\rangle=\langle Fg,h\rangle.$$
As a result, if $\check g_i$, $i=1,\ld,\mu$ denotes the $\OO_S$ basis of $\OO_\SS$ dual 
to the basis $g_i=\p F/\p u_i$, $i=1,\ld,\mu$, then replacing $FdF(\p/\p u_i)$ in \eqref{tba} by $\check g_i$, we get
generators $\chi_1,\ld,\chi_\mu$ for $\TTT$ whose matrix of coefficients with respect to the
$\p/\p u_j$ is the symmetric matrix with $i,j$ entry $\chi_{ij}=\langle \check g_i,F\check g_j\rangle$.

In our calculations in section 7 we always use such a basis. 
We note that if $\o_1,\ld,\o_\mu$ is the basis for $\O^1(\log D)$ dual to $\chi_1,\ld,\chi_\mu$ then 
\beq\label{useq}k^{-1}\beta(\o_i)=\frac{\p}{\p u_i}, \quad\text{and}\quad k^{-1}\a(du_i)=\chi_i.\eeq
\section{ The Gau{\ss}-Manin connection}
The study of the Gau{\ss}-Manin connection for hypersurface singularities 
was initiated by {\sc Brieskorn} in \cite{brieskorn} and has since then 
developed into a very detailed theory. We can only outline the parts of the theory  that are relevant to our 
application. For a more detailed accounts we refer to \cite{greuel}, \cite{looijenga}, \cite{agv}, \cite{kulikov}, \cite{hertling} 
and the original papers quoted there.
 
\subsection{The cohomology bundle and its extensions}
The spaces $H^n(X_u)=H^n(X_u;\CC)$ fit together into the cohomology bundle 
\[H^* =\bigcup_{u \in S \ssm D} H^n(X_u)\] 
over $S\ssm D$. It is a flat vector bundle and the associated sheaf of holomorphic sections
\[\s H^* =H^* \otimes_\CC \OO_{S \ssm D} \]
is equipped with a natural flat connection, the Gauss Manin connection,
\beq\label{gm}\nabla:\s H^*\to \s H^*\otimes_{\OO_S}\O^1_{S\setminus D}\eeq
As usual, we write 
\[ \nabla_{\vartheta}: \s H^* \too \s H^*\]
for the action of a vector field $\vt \in \Theta_{S \ssm D}$. 
The sheaf $\s H^*$ over $S\ssm D$ has  various extensions to $S$. Most relevant to us is the parameterised version of Brieskorn's module $H'$:  
\beq\label{brieskorn} 
\s H':=\pi_*(\O^n_{X/S})/d\pi_*(\O^{n-1}_{X/S}).
\eeq
A section of $\s H'$ over $U \subset S$ is represented by a (relative) holomorphic $n$-form $\eta$ on $\pi^{-1}(U) \subset X$. If $U \subset S \ssm D$ and $u \in U$, the restriction of $\eta$ to the smooth fibre $X_u$ is a closed form $n$-form
and thus determines a cohomology class
\[ [\eta|_{X_u}] \in H^n(X_u)\] 
In this way one obtains an isomorphism $\s H' (U) \to \s H^*(U)$ and thus $\s H'$ can be considered as an {extension} of $\s H^*$, that is, there is a map of $\OO_S$-modules
\[ \s H' \too j_* \s H^*,\;\;\;\]
which is an isomorphism over $S \ssm D$, where $j:S \ssm D \into S$ is the inclusion.
The sheaf $\s H'$ is a locally free sheaf of rank $\mu$, but for a general $\vartheta \in \Theta_S$
the Gau{\ss}-Manin connection maps $\s H'$ into a bigger extension $\s H'' \supset \s H'$. This
second Brieskorn module $\s H''$ can be defined as
\[ \s H'' :=\pi_* \omega_{X/S}/d \pi_*(d\Omega^{n-1}_{X/S}),\]
where $\omega_{X/S}$ denotes the relative dualising module, \cite[page 158]{looijenga}. Elements
from $\omega_{X/S}$ are most conviently described as residues of $n+1$-forms, that is, as
Gelfand-Leray forms. There is an exact sequence
\beq\label{bries} 0 \too \s H' \too \s H'' \too \Omega^{n+1}_{X/S} \too 0 \eeq

When we restrict to logarithmic vector fields, the connection maps $\s H'$ and $\s H''$ to 
themselves, so we have logarithmic connections
$$\nabla:\s H'\too \s H'\otimes_{\OO_S}\O^1_S(\log D)$$
$$\nabla:\s H''\too \s H''\otimes_{\OO_S}\O^1_S(\log D)$$
extending the Gauss-Manin connection \eqref{gm}.
(As there is no possibility of confusion, we denote all these maps by the same symbol $\nabla$)

The action of  $\chi \in \Theta_S(-\log D)$ on a local section $[\eta ]$ represented by a 
relative $n$-form $\eta$ is given by the Lie derivative with respect to a lift 
$\wt \chi$ of $\chi$:
\[ \nabla_{\chi} \eta =[Lie_{\wt \chi}(\eta)]\]
(\cite[p. 148]{looijenga}).

\subsection{$\s H'$ and the cohomology of singular fibres}
We have seen that for $u \in S\setminus D$, the restriction of a global relative $n$-form 
$\eta$  to a smooth fibre $X_u$ determines a cohomology class
\[ [\eta|_{X_u}] \in H^n(X_u)\] 
If $u \in D$ then the fiber $X_u$ is singular, but the form $\eta$ still can be integrated 
over $n$-cycles in $X_u$ and gives rise to a well defined cohomology class in $H^n(X_u)$. 
We sketch the argument. Suppose  $\g_1$ and $\g_2$ are $n$-cycles in
$C_u$ and  $\G$ is a $n+1$-chain in $X_u$ with $\p \G=\g_1-\g_2$. After subdivision, we can write $\G=\G'
+\G''$ where $\G'$ is a $n+1$-chain in the smooth part of $C_u$
and $\G''=\G\cap \bigcup_i B_\ve(p_i)$, where the $p_i$ are the singular points of $C_u$.  Then 
$$\int_{\g_1}\eta-\int_{\g_2}\eta=\int_{\p\G'}\eta+\int_{\p\G''}\eta.$$
The first integral on the right hand side vanishes by Stokes's Theorem. The contribution
 $\int_{\p\G''}\eta$ tends to $0$ as $\ve\to 0$, as the integrand is regular and one is integrating over ever smaller sets.\\

In general, if $Z$ is any analytic space with singularities we can look at the de Rham complex
 $(\Omega_Z^{\bullet},d)$ of K\"ahler forms, and integration over $p$-cycles is well-defined and determines a {\em de Rham evaluation map}
\[ DR: H^p(\Gamma(Z,\Omega_Z^{\bullet}))\to H^p(Z,\CC) \]
If $Z$ is a Stein space, then this map is even {\em surjective}. 
The reason is the following: because $Z$ is Stein, the group at the left hand
side is equal to the $p$-th hypercohomology group $\HH^p$ of the de Rham complex 
$(\Omega^{\bullet}_Z,d)$. The map of complexes $\CC_Z \to (\Omega^{\bullet}_Z,d)$
(induced by the inclusion map $\CC_Z \to \OO_Z$) induces a map
\[\alpha: H^p(Z;\CC)=\HH^p(\CC_Z) \to \HH^p((\Omega_Z^{\bullet},d)) = H^p(\Gamma(Z,\Omega_Z^{\bullet}))\]
and it is shown in \cite{looijenga}, p.141, that $DR$ is a {\em section} of the
map $\alpha$, i.e. $DR \circ \alpha =Id$. In particular, $DR$ is surjective.

Proposition $(8.5)$ of \cite{looijenga} provides a relative version of this
argument, that we specialise to our situation of $\pi:X \too S$.
For this we look at the (truncated) relative de Rham complex
\[ 0 \too \OO_X \too \Omega^1_{X/S} \too \ldots \too \Omega^{n-1}_{X/S} \too \Omega^n_{X/S}\]
The  cohomology sheaves are $\pi^{-1}\OO_S$ in degree $0$ and 
\[ {\mathcal H}^n:=\Omega^{n}_{X/S}/d\Omega^{n-1}_{X/S},\]
a sheaf supported on $\wt D$, in degree $n$.
The direct image $(\pi_*\Omega_{X/S}^{\bullet},d)$ also has two non-vanishing
cohomologies, namely $\pi_*\pi^{-1}\OO_S$ in degree $0$ and $\s H'$ in degree $n$. 
The two hypercohomology spectral sequences now produces a short exact sequence
\[ 0 \too R^n\pi_*(\CC_X) \otimes \OO_S \stackrel{\alpha}{\longrightarrow} \s H'
\stackrel{\beta}{\longrightarrow} \pi_*{\mathcal H}^n \too 0 \]
(\cite[Proposition 8.5]{looijenga}).
Restriction to a (geometrical) fibre over $u$ gives an exact sequence
\[ 0 \too H^n(X_u) \too \s H'|_u \too \pi_*{\mathcal H}_u^n \too 0\]
In the middle we have a vector space of dimension $\mu$, at the right hand side
a direct sum of vector spaces of dimension $\mu_i$, the Milnor numbers of the
singularities appearing in the fibre over $u$. So indeed
\[ \dim H^n(X_u)=\mu -\sum \mu_i\]
The composition 
\[ R^n\pi_*(\CC_X) \too R^n\pi_*(\CC_X) \otimes \OO_S \stackrel{\alpha}{\too}
\s H'\]
is for any $u \in S$ a {\em section} to the deRham-evaluation map
\[DR_u: \s H'_u \too H^n(X_u,\CC)\]

\begin{corollary}
For all $u\in S$, the deRham evaluation map 
\[\s H'_u \to H^n(X_u); \eta \to [\eta|_{X_u}]\] 
is surjective.
\end{corollary}

\subsection{The period map}
The theory of the period map was developed independently by {\sc Varchenko} and {\sc K. Saito}
around the same time and has numerous applications. The basic idea is quite simple.
Let us first fix a relative $n$-form $\eta$ and a point $u \in S \ssm D$ and a horizontal 
basis $\gamma_1(s),\gamma_2(s),\ldots,\gamma_{\mu}(s) \in H_n(X_s)$ for points $s$ in a 
neighbourhood $U$ of $u$. The {\em period map} 
\[P_{\eta}: U \too \CC^{\mu},\;\;\;s \mapsto \left(\int_{\gamma_1(s)}\eta,\int_{\gamma_2(s)}\eta,\ldots,\int_{\gamma_{\mu}(s)} \eta \right)\]
send a point $s$ to the tuple of periods of the form $\eta$. By further parallel transport one
extends $P_\eta$ to a (multi-valued) map 
\[P_{\eta}: S \ssm D \too \CC^{\mu}\]
between spaces of the same dimension $\mu$. The form $\eta$ is called  {\it non-degenerate} if it is non-degenerate at all points $u\in S\ssm D$,  which means that 
$P_{\eta}$ is a local isomorphism near $u$. Of course, this can be tested by looking at the
derivative of this map, which can be identified with the map
\[ \nabla P_{\eta,u}: T_u S\to H^1(X_u),\;\;\vartheta \mapsto [\nabla_{\vartheta} \eta|X_u] \in H^n(X_u) \] 
which is the geometrical fibre at $u$ of the sheaf map
\[ \Theta_{S \ssm D} \too \s H^*,\;\;\;\vartheta \mapsto \nabla_{\vartheta}\eta\]
This map extends to a sheaf map
\[\Theta_S \too \s H'',\;\;\; \vartheta \mapsto \nabla_{\vartheta}\eta\]
which is an {\em isomorphism} in case $\eta$ is non-degenerate. 

\begin{proposition}
A non-degenerate relative $n$-form $\eta$ gives rise to a commutative diagram
$$\xymatrix{0\ar[r]&\s H'\ar[r]&\s H''\ar[r]&\O^{n+1}_{X/S}\ar[r]&0\\
0\ar[r]&\Theta_S(-\log D)\ar[r]\ar[u]&\Theta_S\ar[r]\ar[u]&\OO_\D\ar[r]\ar[u]&0}
$$
with exact rows and where the vertical maps are the isomorphisms described in the last paragraph and where the map
at the right hand side is induced by multiplication by $\o=d\eta$.
\end{proposition}
This diagram can be found in \cite[p. 1248]{Sai83}. 

From this we get immediately the following

\begin{theorem}\label{periodsurject} If $\eta$ is non-degenerate, then for each point $u \in S$ one
obtains an isomorphism
\[ \nabla P_{\eta,u}: T^{log D}_uS \too \s H'_u \]
The composition with the de Rham evaluation map gives a surjection
\[ \text{DR}\circ\nabla P_{\eta,u}: T^{log D}_uS \too H^n(X_u)\]

\end{theorem}
In fact the restriction of  $\text{DR}\circ\nabla P_{\eta,u}$ to 
$T_u\s R(u)$ 
is an isomorphism. 
This  statement was shown by Varchenko to hold in special cases and conjectured to
hold in general, \cite{Var87}. A proof basically along these lines
was sketched to us in a letter by {\sc Hertling}, \cite{Hert}. 

\section{The case of curves}
We specialise to the case $n=1$, so $C:=X_0$ is a plane curve singularity. If $C$ has $r$ branches then by the formula of {\sc Milnor} 
\[\mu=2\delta-r+1,\] 
and for $u \in S\ssm D$ the fibre $C_u:=X_u$ is a smooth Riemann surface
of genus $\delta-r+1$ with $r$ boundary circles. For $u\notin D$, $C_u$ is a smooth Riemann surface of genus $\delta$. For $u \in D$ the curve $C_u$ is singular, say with
singularities $(C_u,p_i)$, $i=1,2,\ldots,N$ and its normalisation $\wt C_u$ has genus
\[ \delta(C)-\delta(C_u)\] 
where $\delta(C_u)=\sum_{i=1}^N \delta(C_u,p_i)$.

\subsection{Intersection form}
Now assume that $C$ is irreducible. For fixed $u\in S$ let $C_u^*=C_u/\p C_u$ be the closed Riemann surface obtained by shrinking $\p C_u$ to a point, and let $\tilde C_u$ and $\tilde C_u^*$ be the normalisations of $C_u$ and $C_u^*$. 

The diagram 
$$ \xymatrix{\wt C_u\ar[r]\ar[d]^n&\widetilde C_u^*\ar[d]^n\\
C_u\;\ar[r]^j&C_u^*}
\hskip .2in\text{gives rise to the diagram}\hskip .2in
\xymatrix{ H^1(\wt C_u)&H^1(\widetilde C_u^*)\ar[l]_{\simeq}\\
H^1(C_u)\ar[u]_{n^*}&H^1(C_u^*)\ar[u]_{n^*}\ar[l]_{\simeq}}
$$
in which the vertical arrows are surjections. Write $I_u$ and $\tilde I_u$ for the intersection forms on
$C_u$ and $\wt C_u$. These are pulled back from the intersection forms on the closed curves $C_u^*$ and $\wt C_u^*$  by means of the isomorphisms in the preceding diagram. Because $n_*: H_2(\wt C_u,\p \wt C_u)\simeq H_2(C,\p C)$, it follows by functoriality that
\beq\label{iso0}\wt I_u(n^*a,n^*b)=I_u(a,b),\eeq
Note that the form $\wt I_u$ is non-degenerate.

\subsection{de Rham version of $I_u$}
The pairing $I_u$ has the following {\sc de Rham} description. 
We choose a a pair of collars $U \subset V \subset C_u$ around the boundary $\partial C_u$ 
and a $C^{\infty}$ bump-function $\rho$, equal to $1$ on $U$ and $0$ outside $V$. If $\eta$ is
a holomorphic (K\"ahler) $1$-form on $C_u$, it follows from Stokes theorem that  
\[ \int_{\partial C} \eta =0\]
By integration we can therefore find a holomorphic function $\alpha$ on $V$ with $d\alpha=\eta$
on $V$. The form $\eta$ is cohomologous to $\wt \eta:=\eta-d \rho \alpha$ and as 
$\rho=1$ on $U$ and there $ d\alpha=\eta$, it follows that $\wt \eta$ is a form
with compact support, contained in $C\setminus U$. It is holomorphic and equal to $\eta$
outside $V$, but only $C^{\infty}$ on the annulus $V \setminus U$.
One then has, using Stokes theorem

\[I_u([\eta],[\eta'])=I_u([\wt \eta],[\eta'])=-\int_{\partial C} \alpha \eta'\]

More details are given in Section 7. 

\subsection{Extension to $\s H^*$ and $\s H'$}

The pairings $I_u$ on $H^1(C_u)$ combine to give a perfect duality
$$I:\s H^*\times\s H^*\to\OO_S$$
over $S\ssm D$. Because of its topological origin, the intersection form is {\em horizontal}
with respect to the Gauss-Manin connection: for any two sections $s_1, s_2$ of 
$\s H^*$, 
$$d\bigl(I(s_1,s_2)\bigr)=I(\nabla s_1,s_2)+I(s_1,\nabla s_2).$$

Using a relative version of the above {\sc de Rham}-description of the intersection pairing
one obtains an extension of $I$, still called $I$, to $\s H'$:
$$I: \s H' \times \s H' \to \OO_S$$ 
For two sections $\eta_1,\eta_2 $ of $\s H'$ one has
\beq\label{keyform} I(\eta_1,\eta_2)(u)=I_u([\eta_1{|C_u}],[\eta_2{|C_u}])\eeq

\subsection{Pulling back the intersection form}\label{pullback}
Using the period map one can pull-back the intersection form on $H^1(C_u)$ to obtain a
$2$-form on $S$.
Let us first start with an arbitrary section $s \in \s H^*$ over $S \setminus D$. 
From it we obtain  a 2-form 
\[\O=s^*I \in \Omega^2_{S \ssm D}\] 
on $S\ssm D$ by the formula
\[ \O(\vt_1,\vt_2):= I(\nabla_{\vt_1} s, \nabla_{\vt_2} s)\]
\begin{proposition}
The form $\O$ is closed.
\end{proposition}
\begin{proof}
This is `clear' as we are pulling back the `constant form $I$', but here is a nice
direct calculation: if 
$a, b$ and $c$ are germs of commuting vector fields on $S$ then 
$$d(s^*I)(a,b,c)=d\bigl(I(a,b)\bigr)(c)-d\bigl(I(a,c)\bigr)(b)+d(\bigl(I(b,c)\bigr)(a)$$
$$= I(\nabla_c\nabla_a s, \nabla_b s)+I(\nabla_a s, \nabla_c\nabla_b s)$$
$$-I(\nabla_b\nabla_a s, \nabla_c s)-I(\nabla_a s, \nabla_b\nabla_c s)$$
\beq\label{close}+I(\nabla_a\nabla_b s, \nabla_c s)+I(\nabla_b s, \nabla_a\nabla_c s)\eeq
Because $a$ and $b$ commute and $\nabla$ is flat, $\nabla_a\nabla_b =\nabla_b\nabla_a $, and 
similarly for $\nabla_a\nabla_c$ and $\nabla_b\nabla_c$. This means that all terms on the right 
hand side in \eqref{close} cancel, except the first and last. These cancel because of the 
anti-symmetry of $I$.
\end{proof}

\begin{theorem}{\em(\cite{gv82})} If $s=\eta$ is a non-degenerate section of $\s H'$, then
 $\O$ is itself non-degenerate and hence symplectic, and moreover extends to all of $S$ as 
a symplectic form. 
\end{theorem}

\section{Results}

In \cite{gv82} one find the formulation of a {\em principle} that the types of degeneration that occur in the
fibres $C_u$ are reflected in the lagrangian properties of the corresponding strata. Our results can be seen
as a vindication of this principle in some special cases.

As before, we will consider the versal deformation $\pi:X \too S$ of an irreducible curve singularity,
a non-degenerate section $\eta$ of the Brieskorn-module $\s H'$ and the resulting symplectic form
$\O$ on $S$, obtained by pulling back the intersection form on the fibres $H^1(C_u)$.  
 
\subsection{The rank of $\O$ on the logarithmic tangent space}

Recall that for a point  $u \in S$, the logarithmic tangent space $T_u^{log D}S  \subset T_uS $ is the 
sub-space spanned by the logarithmic vector fields at $u$. 

\begin{theorem}\label{rank}
The rank of $\O$ restricted to $T_u^{log D}S$ is equal to the rank of $I_u$ on $H^1(C_u)$,
which is equal to $\dim H^1(\wt C_u)=2(\delta(C)-\delta(C_u))$.
\end{theorem}
\begin{proof}
Let $\s R(u)$ and $\s K(u)$ denote, respectively, the right-equivalence stratum and the $\s K$-equivalence 
stratum containing $u$. Recall that by \ref{periodsurject} the period map maps the space $T_u\s K(u)$ 
surjectively to $H^1(C_u)$; its restriction to $T_u\s R(u)\subseteq T_u\s K(u)$ maps isomorphically to $H^1(C_u)$.  
From \eqref{keyform} it follows that the rank of $\O$ on $T^{\log D}_uS$ at $u$ is equal to the rank of the 
intersection form $I_u$ on $H^1(C_u)$, which is equal to the rank of $H^1(\wt C_u)$, and therefore is equal to 
$\mu(C)-2\delta(C_u)=2\delta(C)-2\delta(C_u)$.
\end{proof}

\subsection{Coisotropicity of the Severi strata}
Recall that a subspace $V$ of a symplectic vector space $(W, \langle\_\,,\_\rangle)$ is 
{\it coisotropic} if $V^\bot\subset V$, where 
$V^\bot=\{w\in W:\langle v,w\rangle =0\ \text{for all }v\in V\}$.
A submanifold $X$ of a symplectic manifold $M$ is coisotropic if for all $x\in X$,
$T_xX$ is a coisotropic subspace of $T_xM$.
A singular subset $X$ of the symplectic manifold $M$ is coisotropic if $X_{\text{reg}}$ is coisotropic.

\begin{theorem}\label{coiso} All the Severi strata 
\[D(\delta)\subset D(\delta-1)\subset\cd\subset D(1)=D\]
are coisotropic with respect to $\O$. 
\end{theorem}
\begin{proof}
Suppose that $u$ is a regular point of $D(k)$, so $C_u$ has exactly $k$ ordinary double points as singularities.
As $\s R(u)=\s K(u)=D(k)$ near $u$, the tangent space $T_uD(k)$ is the same as $T^{log D}_uS$. 
From theorem \ref{rank} the rank of $\Omega_{|T_uD(k)} $ is equal to $\mu-2k$, hence $\dim\ker \Omega_{|T_uD(k)} =k$. 
But from the non-degeneracy of $\O$ it follows that $T_uD(k)^\bot$ has dimension
equal to the codimension of $D(k)$, 
namely $k$.   Thus both sides in the relation
$$T_uD(k)^\bot\supset \ker (\O_u|_{T_uD(k)})$$
have dimension $k$, and are therefore equal. It follows that $T_uD(k)^\bot\subset T_uD(k)$.  That is, $D(k)$ is 
coisotropic.
\end{proof} 

The principle mentioned above explains this result by simply saying the near a regular point $u \in D(k)$ 
there are $k$ mutually non-intersecting cycles vanishing at $u$, which make up an isotropic subspace of  $H_1$. However, making this into 
an honest proof is another matter and leads to the considerations outlined above. The form
$\O$ is not unique, and moreover is determined globally rather than locally. One cannot prove anything by using a local normal form $N(\ell):=\{u_1\cd u_\ell=0\}$ for $D$ at a generic point $u_0$ of a Severi stratum $D(\ell)$, since the symplectic form one picks there will not in general coincide with the pullback of the form $\O$ by an isomorphism identifying $(D,u_0)$ with $(N(\ell),0)$.

\subsection{Equations for the $D(k)$}\label{equator}
Let $\chi_1,\ld,\chi_\mu$ be a basis for for $\Theta_S(-\log D)$, and  let $\o_1,\ld, \o_\mu$ be the dual basis for $\O^1_S(\log D)$. Considering $\O$ as an element of $\O^2_S(\log D)$, it can be expressed as the sum
$$\O=\sum_{i<j}\O(\chi_i,\chi_j)\o_i\wedge \o_j.$$
We denote the skew matrix with $i,j$'th entry $\O(\chi_i,\chi_j)$ by $\chi^t\O\chi$.
Then 
\beq\label{formid}\wedge^k\O=\sum_{1\leq i_1<\cd<i_{2k}\leq \mu}\Pf (\chi^t\O\chi(i_1,\ld,i_{2k}))\o_{i_1}\wedge\cd\wedge \o_{i_{2k}}
\eeq
where $\chi^t\O\chi(i_1,\ld,i_{2k})$ is the submatrix of $\chi^t\O\chi$ consisting of rows and columns $i_1,\ld,i_{2k}$ and 
$\Pf$ denotes its Pfaffian. 
The ideal generated by the coefficients of $\wedge^k\O$ with respect to the basis $\o_{i_1}\wedge \cd\wedge \o_{i_{2k}}$ of $\O^{2k}(\log D)$ is the same as the ideal $\Pf_{2k}(\chi\O\chi)$ of $2k\times 2k$ Pfaffians of $\chi^t\O\chi$. 
\begin{theorem}\label{th111} $D(k)=V\left(\Pf_{2(\delta-k+1)}(\chi^t\O\chi)\right)$. In particular, the $\delta$-constant stratum $D(\delta)$ is defined by the entries of $\chi^t\O\chi$. 
\end{theorem}
\begin{proof} 
Consider an arbitrary $u \in S$. The rank of the matrix $\chi^t\O\chi$ at $u$ is the rank of $\O$ 
restricted to the space of evaluations at $u$ of the vector fields in $\Theta_S(-\log D)_u$, which is precisely 
the logarithmic tangent space $T_u^{log D}S$. Theorem \ref{rank} states that the rank of $\O$ on $T^{\log}_uD$ at 
$u$ is equal to $2\delta(C)-2\delta(C_u)$.
As the rank of a skew-symmetric matrix is always even and equal to the size of the largest non-vanishing Pfaffian, 
it follows that $D(k)$ is precisely cut out by the Pfaffians of size  $2(\delta-k+1)$ of the matrix $\chi^t\O\chi$,
i.e. $D(k)=V\left(Pf_{2(\delta-k+1)}(\chi^t\O\chi)\right)$.
\end{proof}

A symplectic form $\O$ on a manifold $S$ gives rise to a Poisson bracket $\{\,\_ \, ,\,\_\,\}$
on the sheaf of functions on $S$, as follows: $\O$ determines an isomorphism 
$\O^1_S\to \Theta_S$ sending a 1-form $\a$ to a vector field $\a^\flat$. Then for functions $f,g$, 
$$\{f,g\}=\O((df)^\flat,(dg)^\flat).$$
The vector field $\chi_f:=(df)^\flat$ is called the {\em Hamiltonian vector field} associated to $f$. If $V \subset S$ is a sub-variety and $I(V) \subset \OO_S$ the ideal of functions vanishing on $V$, then
it is easy to show that for a regular point $x \in V$ one has
\beq\label{poie} T_xV^\bot=\{\chi_f(x):f\in I(V)_x\}.\eeq


The following is well-known:
\begin{proposition}
$V \subset S$ is coisotropic if and only if the ideal $I(V)$ is Poisson-closed: 
\[\{I(V),I(V)\}\subset I(V).\]
\end{proposition}
For the convenience of the reader we include a proof.
\begin{proof} 
Let $x \in V$ be a regular point, $v,w \in T_xV^\bot$, and $f,g \in I(V)$ two functions with $\chi_f(x)=v$,
$\chi_g(x)=w$ (using \eqref{poie}). Then
\[ \Omega(v,w)=\Omega(\chi_f(x),\chi_g(x))=\{f,g\}(x)\]
From this we see that $\{f,g\}$ vanishes at $x$ if and only if $\Omega(v,w)=0$, which means that
$T_xV^\bot \subset (T_xV^\bot)^\bot=T_xV$, that is, $V$ is coisotropic.
\end{proof}

Thus, for each of the Severi strata $D(k)$, the ideal $I(D(k))$ is involutive. 
But note that an ideal defining a coisotropic subvariety is not necessarily involutive;
the proof only shows that this holds if the ideal is radical.

\begin{conjecture}\label{conj2} For all $k=1,2,\ldots,\delta$  
(a) $\Pf_{2k}(\chi^t\O\chi)$ is involutive.
(b) $\Pf_{2k}(\chi^t\O\chi)$ is radical.
\end{conjecture}

By the theorem \ref{th1},  $(b) \implies (a)$, as vanishing ideals of coisotropic varieties are involutive. Nevertheless, involutivity of the ideals  $\Pf_{2k}(\chi^t\O\chi)$ may hold even without their being radical. 
\newline
{\bf Problem:} How to write the Poisson bracket of two Pfaffians of $\chi^t\O\chi$ as a linear combination of Pfaffians?
Is there a universal formula?


\section{The symplectic form as Extension}
The matrix $\chi^t\O\chi$ can be considered as an endormorphism of $\OO_S^\mu$ and its cokernel $N_{\O}$ 
defines a rank $2$ Cohen-Macaulay module on $\OO_D$. If the basis $\chi$ of $\Theta_S(-\log D)$ is chosen to be {\it symmetric}, as described in Subsection \ref{symiso}, then $N_\O$ sits in an exact sequence
 \beq\label{firstext}\xymatrix{0&\ar[l]\OO_\D&\ar[l]N_\O&\ar[l]\OO_\D&0\ar[l]}\eeq
In fact, we show that the extension \eqref{firstext} has a coordinate-independent meaning, depending only
on the choice of $\o$ used in the definition of the period map. 
As such it represents an element in the $\O_{\D}$-module
\[ \Ext^1_D(\OO_\D,\OO_\D)\]
and therefore an infinitesimal deformation of $\OO_\D$ as $\OO_D$-module.
We refer to $N_\O$ as the {\it intersection module}. 
For a vector field $\vartheta$, let $\vartheta^\#$ denote the contraction of
$\Omega$ by $\vartheta$.
Begin with the exact sequence
\beq\label{fses}0\leftarrow \frac{\Omega^1_S(\log D)}{\Omega_S^1}\leftarrow\frac{\Omega^1_S(\log D)}
{\Theta_S(-\log D)^\#}\leftarrow \frac{\Theta_S}{\Theta_S(-\log D)}\leftarrow 0.
\eeq
This exists for every divisor $D$ and non-degenerate 2-form $\O$.
Here the first arrow is induced by contraction with $\Omega$, which maps
$\Theta_S$ to $\Omega^1_S$ and $\Theta_S(-\log D)$ to $\Theta_S(-\log D)^\#$.
Then the exact sequence we consider is obtained from \eqref{fses}
by composing the last arrow with 
the isomorphism $k^{-1}\circ \beta:\Omega^1_S(\log D)\to\Theta_S$ 
described in Subsection \ref{symiso}, 
inducing an isomorphism
$$ \frac{\Theta_S}{\Theta_S(-\log D)}\leftarrow\frac{\Omega^1_S(\log D)}{\Omega_S^1}.$$
Since we have a canonical isomorphism $\Theta_S/\Theta_S(-\log D)\to \OO_\D$ defined by 
$dF$, we obtain the exact sequence \eqref{firstext}.
Thus provided the pairing on $\OO_{\Sigma}$ is chosen canonically, the extension class of \eqref{firstext} depends only on $F$ and on the symplectic form.

\begin{remark}{\em If we apply $k^{-1}\circ\beta$ also to the middle term of the sequence \eqref{fses}
as well as the third, we obtain \eqref{firstext} in the slightly different form
\beq\label{2ndext}0\leftarrow\frac{\Theta_S}{\Theta_S(-\log D)}\leftarrow\frac{\Theta_S}
{k^{-1}\circ\beta(\Theta_S(-\log D)^\#)}\leftarrow \frac{\Theta_S}{\Theta_S(-\log D)}\leftarrow 0.\eeq
Note that $k^{-1}\circ\beta(\Theta_S(-\log D)^\#)$ is generated over $\OO_S$ by vector fields whose components with respect to the usual basis $\p/\p u_1,\ld,\p/\p u_\mu$ are given by the columns of the matrix $\chi\Omega\chi$.
It is interesting that in all of the examples where we have made the calculations,  $k^{-1}\circ\beta(\Theta_S(-\log D)^\#)\subset \Theta_S$  {\it is a Lie sub-algebra}, evidently contained in $\Theta_S(-\log D)$. We cannot at present prove this or explain it. }
\end{remark}
\subsection{Calculation of $\Ext$ groups}
We state without proof the results of some relatively straightforward calculation of $\text{Ext}$ groups. Let $\s C$ denote the conductor ideal of the projection $n=\pi|:\D\to D$. 
\begin{lemma}\label{extcalc1}(i) Both $\Ext^1_D(\OO_\D,\OO_\D)$ and $\Ext^2_D(\OO_\D,\OO_\D)$ are 
$\OO_\D/\mathscr{C}$-modules.\vs
(ii) \vskip -30pt $$\Ext^1_{\OO_D}(\OO_{\D},\OO_{\D})\simeq 
\frac{\{\OO_{\tilde D}\text{-syzygies of }g_1,\ld, g_\mu\}}
{\OO_\D\cdot\{\OO_D\text{-syzygies of }g_1,\ld, g_\mu\}};$$
\vskip 10pt 
(iii) 
\vskip -30pt $$\Ext^2_{\OO_D}(\OO_{\D},\OO_{\D})\simeq \OO_{\D}/\mathscr{C}.$$
\end{lemma}
\begin{proposition}\label{extcalc2} $\Ext^1_D(\OO_\D,\OO_\D)$ is a maximal Cohen-Macaulay module over $\OO_{\D}/\mathscr{C}$
presented by the matrix $\widetilde \chi$ obtained from the symmetric matrix $\chi$ of the basis for 
$\Theta_S(-\log D)$ by deleting its last row and column. 
\end{proposition}

In \cite{MR3009540} it is shown that if $n:\D\to D$ has corank 1 then $\coker\widetilde\chi
\simeq \pi_*\OO_{D^2(n)}$, where, by $D^2(n)$, we mean the double-point scheme of the map $n$:
$$D^2(n)=\text{closure}\{(x_1,x_2)\in\D\times\D:x_1\neq x_2, n(x_1)=n(x_2)\}.$$
The isomorphism holds only if $n$ has corank $1$. The map $n:\D\to D$, normalising the discriminant in the base of a versal deformation,  has corank $1$ exactly for the $A_\mu$ singularities.  Thus, for
the $A_\mu$, and only for these, $\Ext^1_D(\OO_\D,\OO_\D)\simeq \OO_{D^2(n)}$.

\section{Computations and Examples}\label{compex}
It was described in \cite{Giv88} how the symplectic form $\Omega$ can be computed in the case of irreducible quasi-homogeneous curve singularities. 
The projective closure of such a curve has a unique point at infinity $\infty$.

\begin{proposition}
Let $C$ be a curve, $\infty \in C$ a smooth point and $\o$, $\eta$ two 
meromorphic differential form, holomorphic on $C \setminus \{\infty\}$.
Then the intersection form of the cohomology classes $[\o], [\eta] \in H^1(C)$
is 
\[I([\o],[\eta])=2 \pi i Res_{\infty}(\alpha \eta)\]
where $\alpha$ is a meromorphic function in a neighbourhood of $\infty$ with
$d\alpha =\o$.
\end{proposition}
\begin{proof}
Choose two small open discs $U \subset V \subset C$ around
$\infty$, and a $C^{\infty}$ bump function $\rho$ on $C$, equal to $1$ on $U$
and $0$ outside $V$. Choose a function $\alpha$ meromorphic on $V$ with $d \alpha=\o$. 
Then $\o -d(\rho \alpha)$ is a $C^{\infty}$ compactly supported form, cohomologous to $[\o]$. Using  $\o \wedge \eta =0$, we find
\[I([\o],[\eta])=-\int_C d(\rho \alpha)\wedge\eta =-\int_U d(\rho \alpha \cdot \eta) \]
and by Stokes theorem
\[ -\int_U d(\rho \alpha \cdot \eta) =-\int_{\partial U} \alpha \eta\]
which, noticing the reverse of orientation in the boundary, gives the
above formula.
 \end{proof}

This proposition can be used to calculate intersections using Laurent-series
exapansions. If the curve $C$ is given by an affine equation $f(x,y)=0$ and
has a single point at infinity, we can find a Laurent parametrisation of $C$
around $\infty$
\[ x(t),y(t) \in \CC[[t]][1/t] \]
If $\o=A(x,y)dx$ and $\eta=B(x,y) dx$ are the differential forms on $C$, then
by substitution we obtain expansions
\[ \o= a(t) dt, \eta=b(t) dt\]
where $a(t),b(t) \in \CC[[t]][1/t]$ are Laurent series. Integrating up one we find
\[\alpha(t)=\int a(t)dt \in \CC[[t]][1/t]\]
and we can compute the cohomological intersection as:
\[I([\o],[\eta])=Res_{0} \alpha(t)b(t)dt\]

\begin{proposition}{\em(\cite{gv82})} Suppose that $f$ is quasihomogeneous. Then for $\o=dx\wedge dy$, the period map $P_\o$ is  non-degenerate.
\end{proposition}
Calculations referred to in the remainder of this section were carried out using {\it Macaulay 2}, \cite{M2}.
\subsection*{Case $A_4$}
We consider the versal deformation of $A_4$ given by 
$$F(x,a,b,c,d)=y^2+x^5+ax^3+bx^2+cx+d.$$
We take the symmetric basis for $\Theta_S(-\log D)$ with Saito matrix \beq
\chi:=\begin{pmatrix}10a&15b&20c&25d\\15b&-6a^2+20c&-4ab+25d&-2ac\\
20c&-4ab+25d&-6b^2+10ac&-3bc+15ad\\
25d&-2ac&-3bc+15ad&-4c^2+10bd
\end{pmatrix}
\eeq
The symplectic form pulled back by the period mapping induced by the 1-form $ydx$ is  
\beq\label{omA4}\Omega=ada\wedge db+ da\wedge dd+3db\wedge dc.\eeq
Therefore the ideal of entries of the matrix $\chi\Omega\chi$, defining the $\delta$-constant stratum $D(2)$,
is  generated by  
\beq\label{idD(2)A4}a^4+\frac{27}{4}ab^2-9a^2c+20c^2-\frac{25}{2}ad, \quad
a^3b+\frac{27}{4} b^3-9abc-10a^2d+50cd\eeq

and 
$$a^3c+\frac{27}{4}b^2c-4ac^2-20abd+\frac{125}{4}d^2$$
\subsection*{Case $A_6$}
A versal deformation of $A_6$ is given by
$$F(x,a,b,c,d,e,f)=x^7+ax^5+bx^4+cx^3+dx^2+ex.$$
We take the basis of $\Theta_S(-\log D)$ with Saito matrix
$$\scriptsize{\left(
\begin{array}{cccccc}
2a & 3b &        4c     &        5d    &          6e   &           7f   \\
3b & -\frac{10}{7}a^2+4c & -\frac{8}{7}ab+5d    &  -\frac{6}{7}ac+6e    &   -\frac{4}{7}ad+7f  &     -\frac{2}{7}ae     \\
4c & -\frac{8}{7}ab+5d  & -\frac{12}{7}b^2+2ac+6e & -\frac{9}{7}bc+3ad+7f &   -\frac{6}{7}bd+4ae  &
    -\frac{3}{7}be+5af \\
5d & -\frac{6}{7}ac+6e& -\frac{9}{7}bc+3ad+7f  & -\frac{12}{7}c^2+2bd+4ae & -\frac{8}{7}cd+3be+5af  & -\frac{4}
{7}ce+4bf \\
6e &-\frac{4}{7}ad+7f & -\frac{6}{7}bd+4ae &    -\frac{8}{7}cd+3be+5af&  -\frac{10}{7}d^2+2ce+4bf &-\frac{5}{7}de
+3cf \\
7f  & -\frac{2}{7}ae  &   -\frac{3}{7}be+5af   &  -\frac{4}{7}ce+4bf   &   -\frac{5}{7}de+3cf   &   -\frac{6}{7}e^2+2df 
\end{array}
\right)}
$$
and symplectic form
$$\O=\begin{pmatrix}
0  &   -3a^2 -c & -6b & 9a &0& -3 \\
3a^2+c & 0  &    -5a & 0 &  -5 & 0\\  
6b  &  5a  &   0 &  -15 & 0 &  0\\  
-9a &  0   &   15 & 0 &  0 & 0\\  
0  &   5   &   0 &  0 &  0 & 0\\  
3  &   0  &    0  & 0 &  0 & 0
\end{pmatrix}
$$  
Each of the ideals $\Pf_{2\ell}$ is Poisson-closed, and defines 
a Cohen-Macaulay variety of codimension $3-\ell+1$.
\vs
3. For $A_8$, each of the ideals $\Pf_{2\ell}$ is Poisson-closed, and defines 
a Cohen-Macaulay variety of codimension $4-\ell+1$. 
\vs  
\subsection*{Case $E_6$ and $E_8$}
A versal deformation of $E_6$ is given by 
$$F(x,y, a,b,c,d,e,f)=x^3+y^4+ axy^2+bxy+cy^2+dx+ey+f.$$
We take the basis of $\Theta_S(-\log D)$ with symmetric Saito matrix $\chi$ equal to 
$$\scriptsize{
\left(\begin{array}{llll}
2a & 5 b & 6c & 8d\\ 

5b &\frac{-a^4}{6}-4ac+8d& \frac{a^2b}{2}+9e &-\frac{a^3b}{12}-\frac{3bc+ae}{2}\\

6c & \frac{a^2b}{2}+9e & -\frac{5b^2+2a^2c+10ad}{3} & +\frac{7ab^2}{12}-\frac{4a^2d}{3}+12f \\

8d & -\frac{a^3b}{12}-\frac{3bc+ae}{2} 
& \frac{7ab^2}{12}-\frac{4a^2d}{3}+12f & -\frac{a^2b^2}{24}+4cd-\frac{7be}{2}+6af \\

9e & \frac{ab^2-a^3c}{6}+\frac{a^2d-9c^2}{3}+12f & \frac{7abc}{6}-\frac{13bd+4a^2e}{3} & \frac{5b^3-a^2bc}{12}-\frac{7abd}{6}-\frac{3ce}{2} \\

12f & \frac{abd}{6}-\frac{a^3e}{12}-\frac{3ce}{2} & -\frac{8d^2}{3}+\frac{7abe}{12}-2a^2f & 
\frac{10b^2d-a^2be}{24}-\frac{4ad^2}{3}-\frac{9e^2}{4}+6cf 
\end{array}\right.}
$$
\beq\label{saie6}\left.\scriptsize{\begin{array}{rr}
9e & 12f\\
\frac{ab^2-a^3c}{6}+\frac{a^2d-9c^2}{3}+12f & \frac{abd}{6}-\frac{a^3e}{12}-\frac{3ce}{2}\\

 \frac{7abc}{6}-\frac{13bd}{3}-\frac{4a^2}{3}e & -\frac{8d^2}{3}+\frac{7abe}{12}-2a^2f\\
 
\frac{5b^3-a^2bc}{12}-\frac{7abd}{6}-\frac{3ce}{2} &

\frac{10b^2d-a^2be}{24}-\frac{4ad^2}{3}-\frac{9e^2}{4}+6cf\\
 \frac{4b^2c}{3}-\frac{a^2c^2}{6}+\frac{8acd-8d^2-5abe-6a^2f}{3}& \frac{bcd}{2}+\frac{5b^2e-a^2ce}{12}+\frac{5ade}{6}-3abf\\

 \frac{bcd}{2}+\frac{5b^2e-a^2ce}{12}+\frac{5ade}{6}-3abf & -\frac{4cd^2}{3}+\frac{11bde}{6}-\frac{a^2e^2}{24}-{b^2f-2adf}
 \end{array}}\right)
 \eeq
The symplectic form $\Omega$ has matrix
\beq\label{omegae6}
\begin{pmatrix}
0& -\frac{1}{15}ab & \frac{1}{5}c & \frac{2}{15}a^2 & 0& \frac{1}{5}\\
\frac{1}{15}ab & 0& 0& 0 &\frac{1}{2} & 0\\
-\frac{1}{5}c & 0& 0& 1 & 0 & 0\\
-\frac{2}{15}a^2 & 0& -1 & 0 & 0 & 0\\
0  & -\frac{1}{2} & 0 & 0 & 0 & 0\\
-\frac{1}{5} & 0  & 0& 0 & 0 & 0
\end{pmatrix}
\eeq
The ideal of $2\times 2$ Pfaffians (i.e. the ideal of entries) of $\chi\Omega\chi$, defining the $\delta$-constant stratum, is Cohen-Macaulay of codimension 3, and Poisson-closed. Below we comment on the computations involved in proving Cohen-Macaulayness.
The ideal $J$ of $4\times 4$ Pfaffians is also Poisson closed, and has codimension 2 but projective dimension 3. 

For both $E_6$ and $E_8$ we check the Cohen Macaulay property for the ideal generated by the entries in the matrix $\chi\Omega\chi$ using the {\it Depth} package of {\it Macaulay 2}, \cite{M2}. To show that this ideal is radical, we use the result of \cite{FGvS}, that the geometric degree of $D(\delta)$ is equal to the Euler characteristic of the compactified Jacobian.  This Euler characteristic is calculated in \cite{Pio07}: for $E_6$ it is 5 and for $E_8$ 7. Using {\it Singular} we computed the algebraic degree of $\OO_{D(\delta)}$, as defined by the ideal of entries of $\chi\O\chi$, and found that it took these values, showing, in view of Cohen-Macaulayness, that this is the reduced structure.  
\vs
\subsection*{Betti numbers of the Severi strata for $A_{2k}$}
The following table shows the non-zero betti numbers of minimal free resolutions of  the ideals
of Pfaffians, $\Pf_{2\ell}$, of the matrix $\chi\Omega\chi$ for singularities of type $A_{2k}$ for 
$1\leq k\leq 4$.

\beq\label{table}\begin{array}{|c||c|c|c|c|}
\hline
&A_2&A_4&A_6&A_8\\
\hline
\ell&\beta_0&\begin{array}{cc}\beta_0&\beta_1\end{array}&\begin{array}{ccc}\beta_0&\beta_1&\beta_2\end{array}&\begin{array}{cccc}\beta_0&\beta_1&\beta_2&\beta_3\end{array}\\
\hline
\begin{array}{c}1\\2\\3\\4\end{array}&
\begin{array}{c}1\\ -\\- \\- \\
\end{array}&
\begin{array}{cc}3&2\\
1&- \\
-&-\\ 
-&-\\
\end{array}&
\begin{array}{ccc}6&8&3\\
5&4&-\\
1&-&-\\
-&-&-\\
\end{array}&
\begin{array}{cccc}10&20&15&4\\
15&24&10&-\\
7&6&-&-\\
1&-&-&-
\end{array}\\
\hline
\end{array}
\eeq
Since depth $+$ projective dimension = dimension $S$ and codim$\,D(j)=j$, it follows from the data in the table that for $A_{2k}$ with $k\leq 4$, each of the rings $\OO_S/\Pf_{2\ell}$, and therefore each of the Severi strata $D(k-\ell+1)=V(\Pf_{2\ell})\subset S$, is Cohen-Macaulay. 
\begin{conjecture} For all $\ell$ and $k$ with $\ell\leq k$, each of the Severi strata $D(\ell)$ in the base of a miniversal deformation of $A_{2k}$ is Cohen Macaulay.\end{conjecture}
\def\cprime{$'$} \def\cprime{$'$} \def\cprime{$'$} \def\cprime{$'$}
  \def\cprime{$'$}
\providecommand{\bysame}{\leavevmode\hbox to3em{\hrulefill}\thinspace}
\providecommand{\MR}{\relax\ifhmode\unskip\space\fi MR }
\providecommand{\MRhref}[2]{%
  \href{http://www.ams.org/mathscinet-getitem?mr=#1}{#2}
}
\providecommand{\href}[2]{#2}

\end{document}